\documentclass[11pt,reqno]{amsart}
\usepackage[letterpaper,margin=1.05in,headheight=14pt]{geometry}
\usepackage[T1]{fontenc}
\usepackage{comment,lmodern}
\usepackage{amsmath,amssymb,amsthm,mathrsfs,mathtools}
\usepackage{enumitem}
\usepackage{microtype}
\usepackage{xcolor}
\usepackage{tikz}
\usepackage[section]{placeins}
\definecolor{linknavy}{HTML}{17365D}
\definecolor{citegreen}{HTML}{1E5B45}
\usepackage[
  colorlinks=true,
  linkcolor=linknavy,
  citecolor=citegreen,
  urlcolor=linknavy,
  pdfauthor={Levi Segal},
  pdftitle={A polynomial circumference bound for tough graphs},
  pdfsubject={Toughness and long cycles in graphs},
  pdfkeywords={graph toughness, circumference, long cycles, Hamiltonicity}
]{hyperref}
\usepackage[capitalise,noabbrev]{cleveref}
\numberwithin{equation}{section}
\setlist{leftmargin=2.1em,itemsep=0.25em,topsep=0.45em}
\allowdisplaybreaks[2]
\newtheorem{theorem}{Theorem}[section]
\newtheorem{conjecture}[theorem]{Conjecture}
\newtheorem{proposition}[theorem]{Proposition}
\newtheorem{corollary}[theorem]{Corollary}
\newtheorem{lemma}[theorem]{Lemma}
\theoremstyle{definition}

\theoremstyle{remark}

\crefname{theorem}{Theorem}{Theorems}
\Crefname{theorem}{Theorem}{Theorems}
\crefname{proposition}{Proposition}{Propositions}
\Crefname{proposition}{Proposition}{Propositions}
\crefname{corollary}{Corollary}{Corollaries}
\Crefname{corollary}{Corollary}{Corollaries}
\crefname{lemma}{Lemma}{Lemmas}
\Crefname{lemma}{Lemma}{Lemmas}
\crefname{definition}{Definition}{Definitions}
\Crefname{definition}{Definition}{Definitions}
\crefname{section}{Section}{Sections}
\Crefname{section}{Section}{Sections}
\crefname{equation}{Equation}{Equations}
\Crefname{equation}{Equation}{Equations}

\crefname{conjecture}{Conjecture}{Conjectures}
\Crefname{conjecture}{Conjecture}{Conjectures}
\DeclareMathOperator{\circum}{circ}
\DeclareMathOperator{\rank}{rank}
\newcommand{\CC}{\mathcal C}

\makeatletter
\def\thm@space@setup{%
  \thm@preskip=12pt plus 2pt minus 2pt
  \thm@postskip=12pt plus 2pt minus 2pt
}
\title[Polynomial length cycles in tough graphs]
{Polynomial length cycles in tough graphs}
\author{Levi Segal}
\address[Levi Segal]{Stanford Online High School, Redwood City, California, USA}
\email{levisegal0@gmail.com}
\date{}
\author{Jacques Verstraete}
\address[Jacques Verstraete]{Department of Mathematics, University of California, San Diego}
\email{jacques@ucsd.edu}
\date{}

\subjclass[2020]{Primary 05C38; Secondary 05C40, 05C45}
\keywords{graph toughness, circumference, long cycles, Hamiltonicity}

\begin{document}
\begin{abstract}
 We prove that every $n$-vertex $15$-tough graph contains a cycle of length $\Omega(n^{1/20})$.
\end{abstract}
\maketitle
\raggedbottom

\section*{AI Declaration Statement}

The authors used GPT-5.6 for optimization of constants, exploratory computation, literature recommendations, and assistance with drafting and revising the manuscript. The authors take responsibility for the mathematical arguments and the final text.

\section{Introduction}\label{sec:intro}

\begin{comment}
A Hamiltonian cycle is a cycle containing every vertex of a graph and deleting a nonempty set of vertices from such a cycle leaves at most as many components as vertices deleted. Chv\'atal \cite{Chvatal} introduced toughness in 1973 to express this necessary condition for Hamiltonicity.
\end{comment}

Chv\'atal \cite{Chvatal} defined the \emph{toughness} of a non-complete graph $G$ by
\[ \tau(G) = \min_{\substack{S \subseteq V(G) \\ \omega(G - S) \geq 2}} \frac{|S|}{\omega(G - S)}\]
where $\omega(G)$ is the number of components of $G$. Complete graphs are defined to have infinite toughness.
A graph $G$ is \emph{$t$-tough} if $\tau(G) \geq t$.
 Chv\'atal \cite{Chvatal} introduced toughness in relation to
Hamiltonian graphs, noting that every Hamiltonian graph is $1$-tough. The following well-known conjecture was proposed by Chv\'{a}tal  \cite{Chvatal}:

\begin{samepage}
\begin{conjecture}\label{conj:chvatal}
There exists a constant $t_0$ such that every $t_0$-tough graph on at least three vertices is Hamiltonian.
\end{conjecture}
\end{samepage}

The conjecture remains open \cite{WangFractional}, and Chv\'atal originally proposed $t_0 = 2$. This was disproved by Bauer, Broersma, and Veldman \cite{BBV}, who constructed, for every $0<\varepsilon<9/4$, a $(9/4-\varepsilon)$-tough graph with no Hamiltonian path, meaning no path containing every vertex. On the positive side, Enomoto, Jackson, Katerinis, and Saito \cite{EJKS} proved the following. A \emph{$2$-factor} of a graph is spanning subgraph in which every vertex has degree two -- equivalently a collection of disjoint cycles covering all vertices.

\begin{theorem}\label{thm:2factor}
Every 2-tough graph with at least three vertices has a 2-factor.
\end{theorem}

Their threshold is also tight: for every $0<\varepsilon<2$, one can show there are $(2-\varepsilon)$-tough graphs without a $2$-factor.  We refer  the  reader to the survey of Bauer, Broersma, and Schmeichel \cite{ToughnessSurvey} for an account of the early developments and to the references \cite{ChenChordal,Hasanvand,JacksonWormald,KabelaKaiser,WangFractional,Win} for related results.

\bigskip

In this paper, we consider the length of a longest cycle when the toughness is fixed. The \emph{circumference} $\circum(G)$ is the length of a longest cycle in $G$, with $\circum(G)=0$ if $G$ has no cycle. Broersma, van den Heuvel, Jung, and Veldman \cite{BHVJ} made the following conjecture (see also Conjecture~2 on page~10 of the survey \cite{ToughnessSurvey}):

\begin{conjecture}\label{conj-main}
    For each $t > 0$, there exists $c_t > 0$ such that every 2-connected $t$-tough $n$-vertex graph contains a cycle of length at least $c_t\log n$.
\end{conjecture}

This conjecture was recently proved by the authors~\cite{SV} with $c_t \geq 2/\log(1 + 1/t) - o(1)$ as $n \rightarrow \infty$. As shown in~\cite{SV}, this is tight in the sense that $c_t = (2 + o(1))t$ as $t \downarrow 0$, though the optimal value of $c_t$ for any fixed $t \leq 1$ is not known. Toughness has also been used together with minimum degree to bound the circumference. Writing $\delta(G)$ for the smallest degree of a vertex, Jung and Wittmann \cite{JungWittmann} proved that if $G$ is an $n$-vertex non-complete $2$-connected $t$-tough, then
$\circum(G)\ge\min\{n,(t+1)\delta(G)+t\}$.
In particular, since $\delta(G) \geq 2t$, this shows that for $n \geq 3$, $n$-vertex graphs of toughness roughly at least $\sqrt{n/2}$ are hamiltonian. At toughness two, a recent construction of Lesniak and Schmeichel \cite{LesniakSchmeichel} shows that a longest cycle can miss a positive proportion of the vertices: for every $\varepsilon>0$, they give a $2$-tough graph $G$ with $\circum(G)<(1/2 +\varepsilon)|V(G)|$.
A polynomial bound lower bound on the circumference holds for 3-connected graphs with bounded maximum degree, as shown by Jackson and Wormald~\cite{JW} (see also  Liu, Yu, and Zhang \cite{LiuYuZhang} and Chen and Yu~\cite{CY} for results on 3-connected cubic graphs). In this paper, we prove a polynomial lower bound for every fixed toughness greater than $14$:

\begin{samepage}
\begin{theorem}\label{thm:main}
For $t > 14$, every $n$-vertex $t$-tough graph $G$ has $\circum(G) = \Omega(n^{9/176})$ as $n \rightarrow \infty$.
\end{theorem}
\end{samepage}

The ideas of the proof were recently used in a breakthrough on vertex-transitive graphs: Ma and Zhao~\cite{MaZhao} 
proved that every $n$-vertex vertex-transitive graph contains a cycle of length $\Omega(n)$. This is a substantial 
advance towards Lov\'{a}sz' Hamiltonicity conjecture~\cite{GLSTY} for vertex-transitive graphs. No analogous polynomial lower bound had previously been obtained for arbitrary graphs from a fixed toughness assumption alone. For $t\ge 266/5$, the exponent improves to $9/131$, as stated in \cref{cor:larget}. We make use of Theorem \ref{thm:2factor} to prove Theorem \ref{thm:main}.
In fact, it is not hard to show that for $t > 1$, a
$t$-tough graph with $n$ vertices contains vertex-disjoint cycles covering at least $(t - 1)n/(t + 1)$ vertices, however, parts of the proof of Theorem \ref{thm:main} do not carry through for $t > 1$.
We propose the following conjecture:

\begin{conjecture}
For each $t > 1$, there exists $\delta_t > 0$ such that every $n$-vertex $t$-tough graph contains a cycle of length at least $n^{\delta_t}$.
\end{conjecture}

If $G$ is the $n$-vertex graph obtained by adding a vertex adjacent to all vertices of the square of a complete binary tree of height $r \geq 3$, then $n = 2^{r+1}$ and one verifies $\tau(G) = t = 1 + 1/(2r-3)$ and
\[ \circum(G) = 2r^2 - 2r + 4 < 2r^2 \leq n^{\delta_t} \]
where $\delta_t = (t - 1)\log_2(1/(t - 1))$ for $1 < t \leq 4/3$. This shows $\delta_t \rightarrow 0$ as $t \downarrow 1$. Perhaps it is true that $\delta_t \rightarrow 1$ as $t \rightarrow \infty$, which is strongly suggested by Conjecture \ref{conj:chvatal}; from the proof of Theorem \ref{thm:main} one only obtains $\delta_t \geq 9/131$ for $t \geq 266/5$. A similar proof gives
a cycle of length $\exp(\Omega(\sqrt{\log n}))$ in every $t$-tough graph with $t > 2$, also using Theorem \ref{thm:2factor}.

\bigskip
\section{Sketch of the proof}

Theorem \ref{thm:2factor} guarantees a 2-factor in every 2-tough graph. The main idea of the proof of Theorem \ref{thm:main} is to recursively combine cycles until we get a very long cycle. A simple but key fact is that if $C$ and $D$ are two vertex disjoint cycles, and the graph contains three vertex-disjoint paths between $V(C)$ and $V(D)$, then two of the paths plus appropriate arcs of $C$ and $D$ give a cycle of length at least $2(|V(C)| + |V(D)|)/3$. Note that graphs of toughness larger than 1 contain such a configuration, since they are 3-connected and we may invoke Menger's Theorem to find the vertex-disjoint paths.

\bigskip

Let us for the moment refer to the union of the cycles plus the three paths as a \emph{gadget} and the cycles $C$ and $D$ as the \emph{gadget cycles}. If $|\mathcal{C}|$ is a family of vertex-disjoint cycles of length at least $L$ in the graph, then we could potentially find $m/2$ vertex-disjoint gadgets, which provide a family of vertex-disjoint cycles of length at least $4L/3$ in the graph. If this could be repeated, starting with a 2-factor with say $m$ components, we would eventually obtain a cycle of length at least $\min\{n/m,(4/3)^{\log_2 m}\} \geq n^{0.3}$. This of course is too ideal, and the main technical part of the proof, given a current collection of disjoint cycles, is to find a relatively large family of vertex-disjoint gadgets with the cycles from $\mathcal{C}$ as gadget cycles.

\bigskip

To do this, we define a multigraph with vertex set $\mathcal{C}$ and with edges representing pairwise vertex-disjoint paths joining cycles $C,D \in \mathcal{C}$; see Section~$\ref{sec:packings}$ and Section~$\ref{sec:trans}$ for more precise definitions. Using the toughness of the graph, one can show that this multigraph has large minimum degree -- see Lemma~$\ref{lem:ports}$. In particular, since the toughness is larger than 14, one can show the multigraph has minimum degree at least 7. This multigraph is useful for finding many disjoint gadgets -- we use Rado's Theorem~\cite{Rado} (also known as the Matroid Transversal Theorem). Thereby, a new large collection $\mathcal{C}'$ of cycles is found. Specifically -- see Lemma \ref{lem:warmup} and Proposition \ref{prop:initial} --
we can always choose $|\mathcal{C}'| = \Omega(n/\circum(G)^{15})$. The main conclusion (see Proposition \ref{prop:round}) is that
\[\frac{\sum_{C \in \mathcal{C}'}|C|}{\sum_{C \in \mathcal{C}}|C|} \geq \bigg(\frac{|\mathcal{C}'|}{|\mathcal{C}|}\bigg)^\rho\]
for some universal constant $\rho \in (0,1)$. We are now able to find a cycle of the required length -- either the process iterates many times, or there is a stage where one of the cycles is already the required length -- see Section \ref{sec:iteration}.

\section{Preliminaries and disjoint transitions}\label{sec:prelim}

\subsection{Definitions and notation}
\label{sec:defns}
All graphs used in this paper are finite, undirected, and simple, meaning that they have no loops or parallel edges except when explicitly mentioned or made clear otherwise. For a graph $H$, write $|H|=|V(H)|$, $e(H)=|E(H)|$, $d_H(v)$ for the degree of $v$, and $\delta(H)$ for its minimum degree, and $\omega(H)$ for the number of components of $H$.

\bigskip

The \emph{length} of a path or cycle in a graph is its number of edges. Disjoint paths or cycles will always mean vertex-disjoint ones unless only internal vertex-disjointness is explicitly stated. An \emph{arc} of a cycle is a subpath of the cycle.

\bigskip

A \emph{theta} consists of three internally disjoint paths with the same pair of ends. The three paths are called the \emph{branches}; in particular, three parallel edges in a  multigraph form a theta.

\bigskip

A \emph{cut vertex} is a vertex whose deletion disconnects a connected graph, a \emph{bridge} is an edge whose deletion increases the number of components, and a \emph{block} is a maximal connected subgraph without a cut vertex. A \emph{cactus} is a graph whose  blocks are bridges or cycles, although it may also have isolated vertices.

\bigskip

For a positive integer $q$, we write $[q]=\{1,\ldots,q\}$.

\bigskip

The external theorems and results used below are the $2$-factor theorem, Theorem \ref{thm:2factor}, Menger's Theorem \cite{Menger}, Mader's path-packing theory \cite{Mader,SchrijverPaths,SchrijverMatroid}, and Rado's Theorem \cite{Rado}. We state the precise forms that we need below.

\bigskip

\subsection{Path packings and independent representatives}
\label{sec:packings}
Let $T\subseteq V(G)$ be a \emph{terminal set}, partitioned into nonempty, pairwise disjoint terminal classes $T_1,\ldots,T_m$. Write $\mathcal T=\{T_1,\ldots,T_m\}$. A \emph{$\mathcal T$-path} has its ends in different terminal classes and no internal vertex in $T$. A \emph{packing} is a family of pairwise vertex-disjoint $\mathcal T$-paths. A \emph{transversal} of these paths is a vertex set meeting every $\mathcal T$-path. The vertices in a transversal may themselves make up some of the terminals fully. We use the following consequence of Mader's Theorem \cite{Mader,SchrijverPaths}: if $\nu$ is the largest possible size of a packing, then a transversal $X$ exists with
\begin{equation}\label{eq:madercover}
 |X|\le2\nu.
\end{equation}

Clearly, deleting this set $X$ would disconnect all the terminal classes from each other in the graph $G$. We also need a matroid associated with the same path system. A \emph{matroid} $M$ on a finite \emph{ground set} $T$ is a family $\mathcal I$ of subsets, called \emph{independent sets}, such that $\varnothing\in\mathcal I$, every subset of an independent set is independent, and whenever $I,J\in\mathcal I$ satisfy $|I|<|J|$, some element of $J\setminus I$ can be added to $I$ while preserving independence. The \emph{rank} of $U\subseteq T$ is
\[
 \rank_M(U)=\max\{|I|:I\subseteq U,\ I\in\mathcal I\}.
\]
In the \emph{Mader matroid}, a set $I\subseteq T$ is independent precisely when some packing has its set of endpoints containing $I$ \cite{SchrijverMader,SchrijverMatroid}; this formulation is also stated explicitly by Pap \cite{Pap}. The packing may also have additional endpoints outside $I$.

\bigskip

For subsets $A_1,\ldots,A_s$ of a matroid's ground set, an \emph{independent system of distinct representatives} consists of distinct elements $a_i\in A_i$ whose set is independent. Note that these sets $A_i$ need not be disjoint. Rado's Theorem \cite{Rado} states that such representatives exist if and only if
\begin{equation}\label{eq:rado}
 \rank_M\left(\bigcup_{j\in J}A_j\right)\ge |J|
 \qquad\text{for every }J\subseteq[s].
\end{equation}

\bigskip

\subsection{Disjoint transitions}
\label{sec:trans}
Let $\CC=\{C_1,\ldots,C_q\}$ be disjoint cycles and let $T=\bigcup_{i=1}^qV(C_i)$. Taking $T_i=V(C_i)$ makes each $\mathcal T$-path a \emph{transition}: a path joining two different cycles whose internal vertices avoid $T$. Its endpoints on the cycles are called \emph{ports}. A packing of transitions defines an \emph{auxiliary multigraph} $Q$ on $[q]$, with one edge $ij$ for each transition from $C_i$ to $C_j$. Parallel edges are retained, and there are no loops. We call the cycles and their prescribed transition paths a \emph{realization} of $Q$. Throughout, these paths are pairwise vertex-disjoint and their internal vertices avoid every original cycle, so all ports are distinct, including those on the same cycle.

\begin{samepage}
\begin{lemma}\label{lem:ports}
Let $b\ge1$ be an integer, let $G$ be $t$-tough, and let $C_1,\ldots,C_q$ be disjoint cycles with $q\ge2$ and $|C_i|\ge L$ for every $i$. If
\begin{equation}\label{eq:portconditions}
 t>2b,\qquad L\ge\max\left\{2b,\frac{2bt}{t-2b}\right\},
\end{equation}
then there is a transition packing whose auxiliary multigraph has minimum degree at least $b$.
\end{lemma}
\end{samepage}

\begin{proof}
Let $M$ be the Mader matroid for the terminal classes $T_i=V(C_i)$. To apply Rado's Theorem to a list containing $b$ copies of each $T_i$, it suffices to prove
\begin{equation}\label{eq:rankcondition}
 \rank_M(T_I)\ge b|I|,
 \qquad T_I=\bigcup_{i\in I}T_i,
 \qquad I\subseteq[q].
\end{equation}
Indeed, a list involving the classes indexed by $I$ and b copies of every class $T_i$ has at most $b|I|$ sets and has union $T_I$.

\bigskip

First suppose $I$ is nonempty and proper, and put $a=|I|$. Merge all classes outside $I$ into one class
\[
 T_0=\bigcup_{j\notin I}T_j.
\]
The new partition has $a+1$ classes, each of size at least $L$, and has the same terminal set $T$. Let $\nu_I$ be the maximum packing size for this partition. By \eqref{eq:madercover}, choose a transversal $X$ of size $x\le2\nu_I$.

Let $r$ be the number of terminal classes with a vertex outside $X$. Deleting an entire class costs at least $L$ vertices, so
\begin{equation}\label{eq:survivingclasses}
 r\ge a+1-\frac{x}{L}.
\end{equation}
No component of $G-X$ meets two different surviving classes. To see this, take a path between vertices of two such classes. Along that path, two consecutive visits to the terminal set must occur in different classes at some point. The subpath between those visits is a $\mathcal T$-path avoiding $X$, a contradiction. A single class may occupy several components, which only strengthens the inequality $\omega(G-X)\ge r$.

\bigskip

If $r\ge2$, toughness gives
\[
 x\ge tr\ge t\left(a+1-\frac{x}{L}\right),
\]
and therefore
\begin{equation}\label{eq:properrank1}
 \nu_I\ge\frac{x}{2}\ge\frac{tL(a+1)}{2(L+t)}.
\end{equation}
If $r\le1$, at least $a$ entire classes lie in $X$, so
\begin{equation}\label{eq:properrank2}
 \nu_I\ge\frac{x}{2}\ge\frac{La}{2}.
\end{equation}
Since the assumptions imply $tL/(L+t)\ge2b$ and $L\ge2b$, both cases give $\nu_I\ge ba$. Every path for the merged partition has an endpoint in $T_I$, so choose one such endpoint from each path of a maximum packing. The chosen endpoints are distinct and form an independent subset of $T_I$ of size at least $ba$, because the same paths are also a packing for the original partition.

For $I=[q]$, use the original partition and let $\nu$ be its maximum packing size. The same argument, now with $q$ classes, gives
\[
 2\nu\ge\min\left\{\frac{tLq}{L+t},\ L(q-1)\right\}\ge bq.
\]
For the last inequality, use $tL/(L+t)\ge2b$ and $L(q-1)\ge2b(q-1)\ge bq$, since $q\ge2$. The full endpoint set of a maximum packing is independent, which proves \eqref{eq:rankcondition} for $I=[q]$; for $I=\varnothing$, the inequality reduces to $0\ge0$.

\bigskip

Rado's Theorem therefore supplies $b$ distinct representatives from each class whose union is independent in $M$. By the definition of this matroid, there is a packing whose endpoint set contains all these representatives, so its auxiliary multigraph has at least $b$ incident transition edges at every vertex since we take $b$ copies of each terminal class which gives $b$ distinct edges for each.
\end{proof}

To obtain the minimum degree seven needed for our later arguments, we take $b=7$ and, for $t>14$, put
\begin{equation}\label{eq:threshold}
 B_t=\max\left\{14,\frac{14t}{t-14}\right\}.
\end{equation}
Once every current cycle has order at least $B_t$, \cref{lem:ports} provides minimum auxiliary degree seven.

\section{Reaching the length threshold}\label{sec:warmup}

We begin with a quantitative form of the elementary construction that
joins two cycles through two disjoint connecting paths. In this section,
the connecting paths need only avoid the two cycles being joined.

\begin{samepage}
\begin{lemma}\label{lem:linklength}
Let $A$ and $B$ be disjoint cycles of orders $a$ and $b$, let $k\ge2$ be an integer, and $m = \lfloor k/2 \rfloor$. Suppose there are $k$ pairwise vertex-disjoint paths from $A$ to $B$, each with interior disjoint from $A\cup B$.
For every integer $j$ with $2\le j\le k$, their union with $A\cup B$
contains a cycle of order at least
\begin{equation}\label{eq:windowbound}
 a+b-\left\lfloor\frac{(j-1)a}{k}\right\rfloor
       -\left\lfloor\frac b j\right\rfloor+2.
\end{equation}
It also contains a cycle of order at least
\begin{equation}\label{eq:pairbound}
 \left\lceil\gamma_k(a+b)+2\right\rceil,
 \qquad
 \gamma_k=
 \begin{cases}
 \dfrac{3m+1}{2(2m+1)},&k=2m+1,\\[10pt]
 \dfrac{3m-2}{2(2m-1)},&k=2m.
 \end{cases}
\end{equation}
\end{lemma}
\end{samepage}

\begin{proof}
The endpoints on $A$ divide it into $k$ arcs between consecutive endpoints. There are $k$ choices of an arc consisting of $j-1$ consecutive such arcs, and each arc occurs in exactly $j-1$ choices. Hence one arc $I\subseteq A$ has length at most $\lfloor(j-1)a/k\rfloor$ and contains exactly $j$ of the marked endpoints, including its two ends. Consider the $j$ connecting paths with endpoints on $I$. Their endpoints on $B$ divide it into $j$ arcs, one of length at most $\lfloor b/j\rfloor$. Choose the two paths ending at the ends of this arc. On $B$, use the complementary arc; on $A$, use the arc complementary to the subpath of $I$ between the two selected endpoints. These arcs and the two connecting paths form a simple cycle.
Since the connecting paths contribute at least two edges, the order
of this cycle satisfies \eqref{eq:windowbound}.

\bigskip

For \eqref{eq:pairbound}, choose an unordered pair of the $k$
connecting paths uniformly at random, and use a longer arc between
their endpoints on each cycle. First consider $A$, which has order $a$. Label the endpoints in
cyclic order and let $x_1,\ldots,x_k$ be the lengths of the arcs
between consecutive endpoints, so that
\[
 x_1+\cdots+x_k=a.
\]
The average length of the shorter arc between two endpoints is
\[
 \Phi(x_1,\ldots,x_k)
 =\frac{1}{\binom{k}{2}}
   \sum_{1\le r<s\le k}
   \min\left\{
       \sum_{i=r}^{s-1}x_i,\,
       a-\sum_{i=r}^{s-1}x_i
   \right\}.
\]
Regard $\Phi$ as a function of nonnegative real variables whose
sum is $a$. Each summand is the minimum of two linear functions,
so $\Phi$ is concave. Moreover, $\Phi$ is unchanged by cyclically
shifting its coordinates, since this simply relabels the endpoints.
Writing $\sigma$ for the cyclic shift
\[
 \sigma(x_1,\ldots,x_k)=(x_2,\ldots,x_k,x_1),
\]
concavity therefore gives
\[
 \Phi(x)
 =\frac1k\sum_{h=0}^{k-1}\Phi(\sigma^h x)
 \le
 \Phi\left(\frac1k\sum_{h=0}^{k-1}\sigma^h x\right)
 =
 \Phi\left(\frac ak,\ldots,\frac ak\right).
\]
Thus the average shorter arc is largest when the consecutive
arc lengths are all equal.

\bigskip

Suppose first that $k=2m+1$. For each $d=1,\ldots,m$, exactly $k$
unordered pairs have a shorter arc consisting of $d$ consecutive
arcs. Consequently,
\[
 \Phi\left(\frac ak,\ldots,\frac ak\right)
 =\frac{k}{\binom{k}{2}}
   \sum_{d=1}^{m}\frac{da}{k}
 =\frac{a}{m(2m+1)}\cdot\frac{m(m+1)}2
 =\frac{a(m+1)}{2(2m+1)}.
\]
Since the shorter and longer arcs together have length $a$,
the average longer arc has length at least
\[
 a-\frac{a(m+1)}{2(2m+1)}
 =\frac{3m+1}{2(2m+1)}\,a
 =\gamma_k a.
\]

If $k=2m$, then for each $d=1,\ldots,m-1$ there are $k$ such
pairs, while there are $m$ pairs whose endpoints divide the
cycle into two arcs of equal length. Hence
\[
 \begin{aligned}
 \Phi\left(\frac ak,\ldots,\frac ak\right)
 &=
 \frac{1}{\binom{k}{2}}
 \left(
   k\sum_{d=1}^{m-1}\frac{da}{k}
   +m\frac{ma}{k}
 \right)\\
 &=
 \frac{a}{m(2m-1)}
 \left(\frac{m(m-1)}2+\frac m2\right)
 =\frac{am}{2(2m-1)}.
 \end{aligned}
\]
The average longer arc therefore has length at least
\[
 a-\frac{am}{2(2m-1)}
 =\frac{3m-2}{2(2m-1)}\,a
 =\gamma_k a.
\]

Applying the same argument to $B$ gives an average contribution
of at least $\gamma_k b$. Although the cyclic orders of the
path endpoints on $A$ and $B$ may differ, choosing a pair of
paths uniformly induces a uniform choice of an endpoint pair
on each cycle. The average sum of the two longer arc lengths
is therefore at least $\gamma_k(a+b)$.
For each selected pair, the two longer arcs and the two
connecting paths form a simple cycle. The connecting paths
contribute at least two edges, so the average order of these
cycles is at least $\gamma_k(a+b)+2$. Some choice attains at
least this average, and its order is an integer. Thus $G$
contains a cycle of order at least $
 \left\lceil\gamma_k(a+b)+2\right\rceil,
$
as required.
\end{proof}

\bigskip

For integers $L\ge k\ge2$, define
\begin{equation}\label{eq:Uk}
 U_k(L)=\max\left\{
 \left\lceil2\gamma_kL+2\right\rceil,
 \max_{2\le j\le k}
 \left(2L-\left\lfloor\frac{(j-1)L}{k}\right\rfloor
             -\left\lfloor\frac L j\right\rfloor+2\right)
 \right\}.
\end{equation}
Each expression $a-\lfloor(j-1)a/k\rfloor$ and
$b-\lfloor b/j\rfloor$ is nondecreasing in its integer argument.
Thus \cref{lem:linklength} gives a cycle of order at least $U_k(L)$
whenever both input cycles have order at least $L$.

\begin{samepage}
\begin{lemma}\label{lem:warmup}
Let $k\ge2$ be an integer and let $G$ be $t$-tough with $t>k-1$.
Put $F=\circum(G)$ and
\begin{equation}\label{eq:Dtk}
 D_{t,k}=1+\frac{k}{t-k+1}.
\end{equation}
If $\CC$ is a family of $q\ge1$ disjoint cycles, each of order at least
an integer $L\ge k$, then there is a family $\CC'$ of disjoint cycles
such that
\begin{equation}\label{eq:warmround}
 |\CC'|\ge\frac{q-1}{D_{t,k}F},
 \qquad |C|\ge U_k(L)\quad(C\in\CC').
\end{equation}
The empty family is allowed when the lower bound is nonpositive.
\end{lemma}
\end{samepage}

\begin{proof}
Choose an inclusion-maximal family $\mathcal M$ of disjoint cycles,
each of order at least $U_k(L)$ and obtained from two members of $\CC$
and $k$ paths as in \cref{lem:linklength}. Put
\[
 U=\bigcup_{K\in\mathcal M}V(K),\qquad
 m=|\mathcal M|,\qquad u=|U|.
\]
Since every cycle has order at most $F$, we have $u\le mF$, and at
most $u$ members of $\CC$ meet $U$; let
$\mathcal R$ be the remaining original cycles.

\bigskip

We separate the components of $G-U$ recursively, marking a cycle when a deleted separator meets it. Whenever a component
contains two unmarked cycles $A,B\in\mathcal R$, there cannot
be $k$ vertex-disjoint paths between them in that component. Indeed,
shortening each such path between a last visit to $A$ and the following
first visit to $B$ gives paths to which \cref{lem:linklength} applies.
The resulting cycle avoids $U$ and could be added to $\mathcal M$.

\bigskip

By the set version of Menger's Theorem~\cite{Menger}, there is therefore
a set of at most $k-1$ vertices meeting every $A$--$B$ path in that
component. Delete this set and mark every cycle of $\mathcal R$ that
it meets; marked cycles are not selected for subsequent separations.
The separator may meet $A$ or $B$, but it cannot delete either cycle
entirely because both have at least $k$ vertices. Their surviving vertices
lie in different components, so each separation increases the component
count by at least one. Repeating this procedure leaves at most one
unmarked cycle in each component.

\bigskip

Let $Z$ be the union of the separators and set
\[
 z=|Z|,\qquad P_0=\omega(G-U),\qquad P=\omega(G-U-Z).
\]
The separators used at different stages are disjoint, so
\begin{equation}\label{eq:separatorcount}
 z\le(k-1)(P-P_0).
\end{equation}
At most $z$ cycles are marked, while each unmarked cycle lies wholly
in a final component. Hence
\begin{equation}\label{eq:remainingcycles}
 |\mathcal R|\le P+z.
\end{equation}
If $P\ge2$, toughness and \eqref{eq:separatorcount} give
\[
 tP\le u+z\le u+(k-1)P,
 \qquad
 |\mathcal R|\le kP\le\frac{ku}{t-k+1}.
\]
If $P=1$, there were no separation stages, so $z=0$ and
$|\mathcal R|\le1$. If $P=0$, then $G-U$ was empty and so was
$\mathcal R$. In every case,
\[
 q\le u+|\mathcal R|\le1+D_{t,k}u\le1+D_{t,k}mF.
\]
Taking $\CC'=\mathcal M$ proves the lemma.
\end{proof}

We use the following fixed sequence of length guarantees:
\begin{equation}\label{eq:lengthsequence}
 \begin{gathered}
 L_0=3,\qquad L_1=6,\qquad L_2=11,\qquad L_3=19,\\
 L_{i+1}=2L_i-\left\lfloor\frac{L_i}{5}\right\rfloor
                 -\left\lfloor\frac{L_i}{4}\right\rfloor+2
 \quad(i\ge3).
 \end{gathered}
\end{equation}
The path counts for these rounds are
\begin{equation}\label{eq:pathsequence}
 k_1=3,\qquad k_2=6,\qquad k_3=11,\qquad k_i=15\quad(i\ge4).
\end{equation}
For the first three rounds, \cref{lem:linklength} gives respectively
$L_1,L_2,L_3$; the second uses $\gamma_6=7/10$, and the third may use
$j=4$. Subsequent rounds use $k=15$ and $j=4$ in
\eqref{eq:windowbound}. Thus $L_i\le U_{k_i}(L_{i-1})$ for every $i$.
Also $k_i\le L_{i-1}$, and the toughness requirement $t>k_i-1$ holds
whenever $t>14$.

For $t>14$, define
\begin{equation}\label{eq:rt}
 r_t=\min\{i\ge0:L_i\ge B_t\},
 \qquad A_t=\prod_{i=1}^{r_t}D_{t,k_i}.
\end{equation}
Both quantities depend only on $t$ and are finite, since
$L_{i+1}\ge31L_i/20+2$ for $i\ge3$ ensures that the sequence
eventually reaches $B_t$.

\begin{samepage}
\begin{proposition}\label{prop:initial}
Let $t>14$, and let $G$ be a $t$-tough graph of order $n\ge3$ and
circumference $F$. With $r_t$ and $A_t$ defined by \eqref{eq:rt}, either
\begin{equation}\label{eq:initiallarge}
 F>\left(\frac{n}{2A_t}\right)^{1/(r_t+1)},
\end{equation}
or $G$ contains a family of $q_0\ge2$ vertex-disjoint cycles, each of
order at least $L_{r_t}\ge B_t$, such that
\begin{equation}\label{eq:initialmany}
 q_0\ge\frac{n}{2A_tF^{r_t+1}}.
\end{equation}
\end{proposition}
\end{samepage}

\begin{proof}
The $2$-factor theorem~\cite{EJKS} gives at least $n/F$ disjoint cycles
covering $V(G)$. Starting from this family, apply \cref{lem:warmup}
with the successive path counts in \eqref{eq:pathsequence}. Let $p_i$
be the number of cycles after $i$ rounds and put
$A_{t,i}=\prod_{h=1}^iD_{t,k_h}$, with $A_{t,0}=1$. Iterating the recurrence
$p_i\ge(p_{i-1}-1)/(D_{t,k_i}F)$ gives
\begin{equation}\label{eq:warmrecurrence}
 p_i\ge\frac{n}{A_{t,i}F^{i+1}}
 -\sum_{j=1}^i\frac{1}{F^{i-j+1}\prod_{h=j}^iD_{t,k_h}}
 >\frac{n}{A_{t,i}F^{i+1}}-\frac12.
\end{equation}
For the last inequality, use $D_{t,k_h}>1$ and $F\ge3$.

\bigskip

Set $T=n/(A_tF^{r_t+1})$. If $T<2$, then \eqref{eq:initiallarge}
holds. If $T\ge2$, the first term on the right side of
\eqref{eq:warmrecurrence} is at least $T$ for every $i\le r_t$.
Thus every intermediate family contains at least two cycles, and the
construction continues through all $r_t$ rounds. At the final round,
$p_{r_t}>T-1/2\ge T/2$, which proves \eqref{eq:initialmany}.
\end{proof}

For $t=15$, the recurrence \eqref{eq:lengthsequence} gives
$L_7=138<210=B_{15}$ and $L_8=217$, so
\begin{equation}\label{eq:15constants}
 r_{15}=8,\qquad L_{r_{15}}=217,\qquad
 A_{15}=\frac{16}{13}\cdot\frac85\cdot\frac{16}{5}\cdot16^5
 =\frac{2^{31}}{325}.
\end{equation}

\bigskip

\section{Smoothing a cactus}\label{sec:cactus}

We first record the elementary structural facts behind the smoothing argument that we use later to recombine cycles without getting stuck or losing any of our large cycles. Note that the definition of a theta also includes three parallel edges, viewed as three paths of length one.

\begin{samepage}
\begin{lemma}\label{lem:cactusbound}
If a loopless multigraph $P$ contains no theta, then every edge-containing block of $P$ is a bridge or a cycle. Isolated vertices are allowed. If $p=|V(P)|$ and $c=\omega(P)$, then
\begin{equation}\label{eq:cactusedges}
 e(P)\le2(p-c).
\end{equation}
\end{lemma}
\end{samepage}

\begin{proof}
Let $B$ be an edge-containing block that is not a single edge. Then $B$ contains a cycle $C$, possibly a cycle of length two. If $B$ has an additional edge with both ends on $C$, this edge and the two arcs of $C$ between its ends form a theta.

\bigskip

Otherwise, suppose $B$ has a vertex outside $C$, and take a component $D$ of $B-V(C)$. Since $B$ is connected, $D$ has a neighbor on $C$. It must have two distinct neighbors there since a unique neighbor would be a cut vertex of $B$. The connectedness of $D$ gives a path between two such neighbors whose internal vertices lie in $D$. This path and the two corresponding arcs of $C$ form a theta. Both possibilities are excluded, so $B=C$.

\bigskip

Within each component of $P$, start from one vertex and attach its blocks successively at their shared cut vertices. This is possible because the block incidence graph is a tree: it has one vertex for each block and each cut vertex, and an edge records containment of a cut vertex in a block. A bridge adds one edge and one new vertex. A cycle block of length $k\ge2$ adds $k$ edges and $k-1$ new vertices, with $k\le2(k-1)$ -- see Figure \ref{fig:cactus-blocks}. In this figure, the gray graph has already been built from $r$. The bold bridge attached at $u$ adds one edge and one new vertex the bold cycle attached at $v$ adds six edges and five new vertices. Summing over all blocks gives at most two edges for each vertex other than the one initial vertex per component, proving \eqref{eq:cactusedges}.
\end{proof}
% Insert after the final paragraph of the proof of Lemma 4.1.
% Requires \usepackage{tikz} in the preamble.
% Gray: the graph already built from r. Black: two blocks being added.
\begin{figure}[htbp]
\centering
\begin{tikzpicture}[x=1cm,y=1cm,line cap=round,line join=round,
  existing/.style={draw=black!50,line width=.5pt},
  added/.style={draw=black,line width=1.1pt}]
  % Begin with r, then add a bridge, a 4-cycle, a bridge,
  % and a 6-cycle. These blocks have already been attached.
  \coordinate (r) at (-3.45,0);
  \coordinate (a) at (-2.65,0);
  \coordinate (b) at (-2.25,.65);
  \coordinate (c) at (-1.85,0);
  \coordinate (d) at (-2.25,-.65);
  \draw[existing] (r)--(a);
  \draw[existing] (a)--(b)--(c)--(d)--cycle;

  \foreach \i in {0,...,5}
    \coordinate (H\i) at
      ({-.10+.85*cos(60*\i)},{.70*sin(60*\i)});
  \draw[existing] (c)--(H3);
  \draw[existing] (H0)--(H1)--(H2)--(H3)--(H4)--(H5)--cycle;

  % A new bridge meets the old graph only at u=H4.
  \coordinate (u) at (H4);
  \coordinate (s) at (-.95,-1.40);
  \draw[added] (u)--(s);

  % A new 6-cycle meets the old graph only at v=H0.
  % Its six vertices are v, K4, K5, K0, K1, K2.
  \coordinate (v) at (H0);
  \foreach \i in {0,1,2,4,5}
    \coordinate (K\i) at
      ({1.65+.90*cos(60*\i)},{.90*sin(60*\i)});
  \draw[added] (v)--(K4)--(K5)--(K0)--(K1)--(K2)--cycle;

  \foreach \name in {r,a,b,c,d,H0,H1,H2,H3,H4,H5}
    \fill[black!60] (\name) circle[radius=.9pt];
  \foreach \name in {s,K0,K1,K2,K4,K5}
    \fill (\name) circle[radius=.9pt];
  \foreach \name in {r,u,v}
    \fill (\name) circle[radius=1.15pt];

  \node[above=3pt] at (r) {$r$};
  \node[below right=-2pt] at (u) {$u$};
  \node[above=4pt] at (v) {$v$};
\end{tikzpicture}
\caption{}
\label{fig:cactus-blocks}
\end{figure}

The following lemma is one of the main lemmas of the paper.

\bigskip

\begin{lemma}\label{lem:smoothing}
Let $\{C_v:v\in V(Q)\}$ be $q\ge1$ disjoint cycles with a realization of a loopless auxiliary multigraph $Q$ satisfying $\delta(Q)\ge7$. Suppose $X\subseteq V(Q)$, write $x=|X|$, and assume that $Q-X$ contains no theta. Then there are pairwise disjoint sets
\[
 S_1,\ldots,S_a\subseteq V(Q)\setminus X,\qquad |S_j|\ge2,
\]
and disjoint cycles $K_1,\ldots,K_a$ in $G$ such that
\begin{equation}\label{eq:smoothingcount}
 R:=\sum_{j=1}^a|S_j|\ge\frac{q-5x+c+2}{4}>\frac{q-5x}{4},
 \qquad c=\omega(Q-X).
\end{equation}
For every $v\in S_j$, the cycle $K_j$ contains an arc $A_v\subseteq C_v$ with at least $2|C_v|/3$ edges. Consequently,
\begin{equation}\label{eq:smoothinglength}
 |K_j|\ge\frac23\sum_{v\in S_j}|C_v|.
\end{equation}
The other parts of $K_j$ use only the prescribed transition paths and arcs on cycles indexed by $X$. In particular, $K_j$ is disjoint from every $C_v$ with $v\notin X\cup S_j$. The output families may be empty when the non-strict lower bound in \eqref{eq:smoothingcount} is nonpositive.
\end{lemma}

\begin{proof}
Call the cycles $C_u$ with $u\in X$ the \emph{host cycles}. Put $P=Q-X$, $p=q-x$, and $c=\omega(P)$. By \cref{lem:cactusbound},
\[
 e(P)\le2(p-c).
\]
We have $x\ge1$: otherwise the minimum degree of $Q$ and the last inequality would give $7q\le2e(Q)\le4(q-c)$, a contradiction.

\bigskip

Let $D$ be the set of vertices of $P$ incident with at least three edges
joining them to $X$, counting parallel edges separately. A vertex outside
$D$ has degree at least five in $P$. Choose a root vertex in each component of $P$ and attach its blocks successively, starting from the root. Each block is attached at one vertex already present and introduces all its other vertices. For a vertex $v$ other than the root, call the block that introduced $v$ its \emph{parent block}; call the blocks subsequently attached at $v$ its \emph{child blocks}. The root has only child blocks.

\bigskip

The parent block contributes at most two to $d_P(v)$. Thus, if $v\notin D$, at least three edges incident with $v$ belong to its child blocks. A child block that is a bridge contributes one such edge and introduces one vertex. A child block that is a cycle contributes two such edges and introduces at least one vertex. The child blocks of $v$ must therefore introduce at least two vertices. The same conclusion holds at a root outside $D$, since all of its incident edges belong to child blocks.

The sets of vertices introduced by the child blocks of different vertices
are disjoint. Altogether there are $p-c$ nonroot vertices, so
$2|V(P)\setminus D|\le p-c$. Hence
\begin{equation}\label{eq:Dsize}
 |D|\ge\frac{p+c}{2}.
\end{equation}

For each $v\in D$, choose three transitions from $C_v$ to cycles indexed by $X$. Their three distinct ports divide $C_v$ into three arcs. Choose a shortest one and let $A_v$ be its complementary arc with the same ends. Thus
\[
 |E(A_v)|\ge\frac23|C_v|.
\]
Let $B_v$ be the path consisting of $A_v$ and the two chosen transitions at its endpoints. The third port lies in the interior of $A_v$, but its transition is omitted, so it introduces no additional edge into $B_v$. Both ends of $B_v$ lie on cycles indexed by $X$; they are distinct vertices even if they lie on the same such cycle. The internal vertices of $B_v$ avoid all cycles indexed by $X$. The disjointness of the transition paths and the original cycles imply that the paths $B_v$, $v\in D$, are pairwise vertex-disjoint.

\bigskip

Define an auxiliary multigraph $H$ with vertex set $X$. For each $v\in D$, add an edge $e_v$ between the indices of the cycles containing the ends of $B_v$. This edge is a loop when both ends lie on the same cycle. Thus $e(H)=|D|$. We use $H$ only to decide which paths to retain and how to connect their endpoints.

\bigskip

First delete at most $x-1$ edges of $H$ so that every remaining degree is even. Choose a spanning tree in each component of $H$ after ignoring loops, and root each tree. For every other vertex, its parent is the next vertex on the unique path to the root. Process these vertices in decreasing order of their distance from the root, measured by the number of edges on that path. If a vertex has odd current degree, delete its tree edge to its parent; otherwise keep that edge. Its degree is then even and no later deletion changes it. The root also has even degree, because the total degree in its original component is even and all other degrees are even. Only edges of the spanning trees have been deleted, at most $x-1$ altogether.

\bigskip

Let $H'$ be the remaining multigraph and put $R_0=e(H')$. Then
\begin{equation}\label{eq:evenedges}
 R_0\ge |D|-x+1.
\end{equation}
Retain only the paths $B_v$ whose edges $e_v$ remain in $H'$. At a vertex $u\in X$, the retained endpoints on $C_u$ are distinct ports, and their number is $d_{H'}(u)$. A loop contributes two ports, in agreement with the degree convention.

\bigskip

Suppose $d_{H'}(u)=2m\ge4$, and list the ports as $z_1,\ldots,z_{2m}$ in cyclic order on $C_u$. There are two consecutive pairings:
\begin{align*}
 &(z_1,z_2),(z_3,z_4),\ldots,(z_{2m-1},z_{2m}),\\
 &(z_2,z_3),(z_4,z_5),\ldots,(z_{2m},z_1).
\end{align*}
For either pairing, join each pair by the arc whose interior contains no other retained port. The selected arcs within that pairing are pairwise vertex-disjoint.

\bigskip

A loop $e_v$ at $u$ will form a cycle containing only its own path $B_v$ precisely when the two ends of $B_v$ are paired together. Any fixed pair occurs in at most one of the two displayed pairings. If $\ell_u$ loops are incident with $u$, choose the pairing for which at most $\ell_u/2$ of them close in this way.

\bigskip

If $d_{H'}(u)=2$, join its two ports by either arc between them. If the two ports belong to one loop, that loop necessarily closes by itself. Let $h$ be the number of degree-two vertices of $H'$ incident with a loop. There is exactly one loop at each such vertex, and $h\le x$. No arcs are selected at vertices of degree zero.

\bigskip

Let $J$ be the subgraph of $G$ formed by the retained paths $B_v$ and all the selected arcs on the cycles indexed by $X$. Each port has one incident edge from its path $B_v$ and one from its selected arc on $C_u$, so its degree in $J$ is exactly two. Every other included vertex lies internally on exactly one of these paths or arcs and also has degree two. There are no other intersections, by the disjointness of the prescribed paths and the choice of disjoint host arcs. Thus $J$ is a finite simple graph all of whose vertices have degree two. Each component is therefore a simple cycle. Figure~\ref{fig:smoothing} shows a component obtained from two of the paths $B_v$.

% A single graph: the bold subgraph is one simple cycle.
% Requires only \usepackage{tikz}.
\begin{figure}[htbp]
\centering
\begin{tikzpicture}[x=1cm,y=1cm,line cap=round,line join=round,
  unused/.style={draw=black!40,line width=.4pt},
  retained/.style={draw=black,line width=1.1pt}]
  \begin{scope}[shift={(-3.2,0)}]
    \foreach \i in {0,...,11}
      \coordinate (U\i) at ({30*\i}:.65);
  \end{scope}
  \begin{scope}[shift={(3.2,0)}]
    \foreach \i in {0,...,11}
      \coordinate (W\i) at ({30*\i}:.65);
  \end{scope}
  \begin{scope}[shift={(0,.85)}]
    \foreach \i in {0,...,11}
      \coordinate (V\i) at ({30*\i}:.7);
  \end{scope}
  \begin{scope}[shift={(0,-.85)}]
    \foreach \i in {0,...,11}
      \coordinate (Z\i) at ({30*\i}:.7);
  \end{scope}

  % All four original cycles are 12-cycles.
  \foreach \name in {U,W,V,Z}{
    \foreach \i [evaluate=\i as \j using {int(mod(\i+1,12))}]
      in {0,...,11}
      \draw[unused] (\name\i)--(\name\j);
  }

  % The selected host arcs join the ports along the outside of each host.
  \draw[retained] (U1)--(U2)--(U3)--(U4)--(U5)--(U6)
    --(U7)--(U8)--(U9)--(U10)--(U11);
  \draw[retained] (W7)--(W8)--(W9)--(W10)--(W11)--(W0)
    --(W1)--(W2)--(W3)--(W4)--(W5);

  % Each selected outside arc has eight of its cycle's twelve edges.
  \draw[retained] (V7)--(V6)--(V5)--(V4)--(V3)--(V2)
    --(V1)--(V0)--(V11);
  \draw[retained] (Z5)--(Z6)--(Z7)--(Z8)--(Z9)--(Z10)
    --(Z11)--(Z0)--(Z1);

  % The four displayed transitions are disjoint paths.
  \coordinate (P1) at (-1.65,.4125);
  \coordinate (P2) at (1.65,.4125);
  \coordinate (P3) at (-1.65,-.4125);
  \coordinate (P4) at (1.65,-.4125);
  \draw[retained] (U1)--(P1)--(V7);
  \draw[retained] (V11)--(P2)--(W5);
  \draw[retained] (U11)--(P3)--(Z5);
  \draw[retained] (Z1)--(P4)--(W7);

  \foreach \name in {U,W,V,Z}
    \foreach \i in {0,...,11}
      \fill (\name\i) circle[radius=.85pt];
  \foreach \i in {1,...,4}
    \fill (P\i) circle[radius=.85pt];
  \node at (-3.2,0) {$C_u$};
  \node at (3.2,0) {$C_w$};
  \node at (0,.85) {$C_v$};
  \node at (0,-.85) {$C_{v'}$};
\end{tikzpicture}
\caption{The bold cycle $K_j$ joins $B_v$ and $B_{v'}$ through the host cycles $C_u$ and $C_w$, with only the selected transitions shown.}
\label{fig:smoothing}
\end{figure}
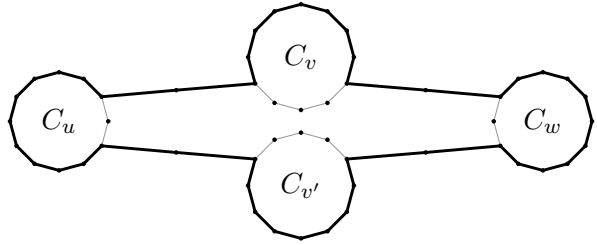

Every component of $J$ contains a path $B_v$. It contains exactly one such path if and only if $e_v$ is a loop whose two ports have been paired together. Discard these components. If $\ell$ is the total number of loops in $H'$, the number discarded is at most
\[
 h+\frac{\ell-h}{2}=\frac{\ell+h}{2}\le\frac{R_0+x}{2}.
\]
Each discarded component contains exactly one retained path. The number $R$ of retained paths in the remaining components consequently satisfies
\begin{align*}
 R&\ge\frac{R_0-x}{2}
 \ge\frac{|D|-2x+1}{2}\\
 &\ge\frac{p+c-4x+2}{4}
 =\frac{q-5x+c+2}{4}
 >\frac{q-5x}{4}.
\end{align*}

Call the remaining components $K_1,\ldots,K_a$ and define
\[
 S_j=\{v\in D:B_v\text{ is a subpath of }K_j\}.
\]
These sets are pairwise disjoint and each contains at least two elements. Since $K_j$ contains the arc $A_v$ for every $v\in S_j$, and these arcs have disjoint edges,
\[
 |K_j|=|E(K_j)|\ge\sum_{v\in S_j}|E(A_v)|
 \ge\frac23\sum_{v\in S_j}|C_v|.
\]
Because every other part of $K_j$ lies on prescribed transitions or on cycles indexed by $X$, the disjointness of the prescribed paths also gives the stated disjointness from all unused original cycles outside $X$.
\end{proof}

\section{Cycle replacement and total length}\label{sec:round}

For a family $\CC$ of disjoint cycles, define
\begin{equation}\label{eq:mass}
 M(\CC)=\sum_{C\in\CC}|C|.
\end{equation}
We use the total length to measure how much of the family survives
each replacement. The aim is to reduce the number of cycles while
retaining a proportion of the total length bounded below by  the ratio of the new and old cycle length counts.

\subsection{Theta replacements}
Consider a theta in an auxiliary multigraph realized by a packing of transitions. Define its \emph{support} to be
the set
\[
    V(\Theta)
\]
of its vertices. Since the vertices of the auxiliary multigraph index the
original cycles, the cycle support of $\Theta$ is the family
\[
    \mathcal S(\Theta)=\{C_v:v\in V(\Theta)\}.
\]
 The two cycles indexed by its common endpoints are called the \emph{branch cycles} and the cycles indexed by the other vertices of a branch are its \emph{internal cycles}. Choose two of the three branches. Replace each auxiliary edge on them by its prescribed transition path, and on every cycle visited join the two selected ports by a longer arc. These paths and arcs form a cycle in $G$. The disjointness of the transitions and original cycles ensures that no vertex is repeated. Retaining also the original internal cycles on the omitted branch gives a family of disjoint cycles.

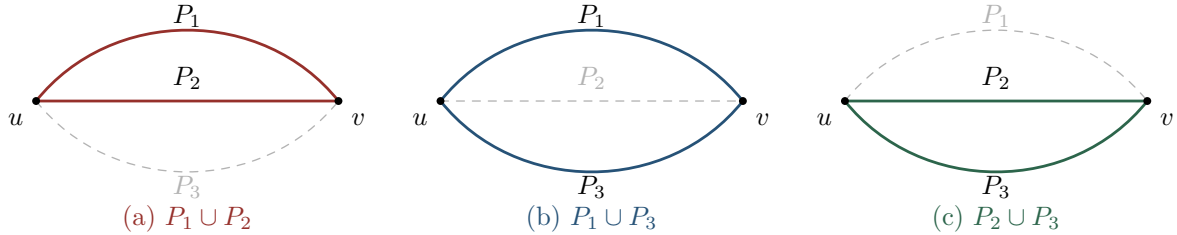
\begin{figure}[htbp]
\centering
\begin{tikzpicture}[
  x=1cm,y=1cm,
  chosen/.style={line width=1.05pt,line cap=round},
  omitted/.style={draw=black!30,line width=.5pt,dashed,line cap=round},
  end/.style={circle,fill=black,inner sep=1.05pt},
  pathlabel/.style={font=\small},
  cyclelabel/.style={font=\small}
]
\definecolor{cycleone}{RGB}{150,48,43}
\definecolor{cycletwo}{RGB}{38,82,119}
\definecolor{cyclethree}{RGB}{44,101,75}

\begin{scope}[xshift=0cm]
  \coordinate (u) at (0,0);
  \coordinate (v) at (4,0);
  \draw[chosen,draw=cycleone] (u) .. controls (1,1.25) and (3,1.25) .. (v);
  \draw[chosen,draw=cycleone] (u) -- (v);
  \draw[omitted] (u) .. controls (1,-1.25) and (3,-1.25) .. (v);
  \node[end] at (u) {};
  \node[end] at (v) {};
  \node[pathlabel,above] at (2,.85) {$P_1$};
  \node[pathlabel,above] at (2,.02) {$P_2$};
  \node[pathlabel,below,text=black!30] at (2,-.85) {$P_3$};
  \node[pathlabel,below left=1pt] at (u) {$u$};
  \node[pathlabel,below right=1pt] at (v) {$v$};
  \node[cyclelabel,text=cycleone] at (2,-1.55)
    {\textup{(a)} $P_1\cup P_2$};
\end{scope}

\begin{scope}[xshift=5.35cm]
  \coordinate (u) at (0,0);
  \coordinate (v) at (4,0);
  \draw[chosen,draw=cycletwo] (u) .. controls (1,1.25) and (3,1.25) .. (v);
  \draw[omitted] (u) -- (v);
  \draw[chosen,draw=cycletwo] (u) .. controls (1,-1.25) and (3,-1.25) .. (v);
  \node[end] at (u) {};
  \node[end] at (v) {};
  \node[pathlabel,above] at (2,.85) {$P_1$};
  \node[pathlabel,above,text=black!30] at (2,.02) {$P_2$};
  \node[pathlabel,below] at (2,-.85) {$P_3$};
  \node[pathlabel,below left=1pt] at (u) {$u$};
  \node[pathlabel,below right=1pt] at (v) {$v$};
  \node[cyclelabel,text=cycletwo] at (2,-1.55)
    {\textup{(b)} $P_1\cup P_3$};
\end{scope}

\begin{scope}[xshift=10.70cm]
  \coordinate (u) at (0,0);
  \coordinate (v) at (4,0);
  \draw[omitted] (u) .. controls (1,1.25) and (3,1.25) .. (v);
  \draw[chosen,draw=cyclethree] (u) -- (v);
  \draw[chosen,draw=cyclethree] (u) .. controls (1,-1.25) and (3,-1.25) .. (v);
  \node[end] at (u) {};
  \node[end] at (v) {};
  \node[pathlabel,above,text=black!30] at (2,.85) {$P_1$};
  \node[pathlabel,above] at (2,.02) {$P_2$};
  \node[pathlabel,below] at (2,-.85) {$P_3$};
  \node[pathlabel,below left=1pt] at (u) {$u$};
  \node[pathlabel,below right=1pt] at (v) {$v$};
  \node[cyclelabel,text=cyclethree] at (2,-1.55)
    {\textup{(c)} $P_2\cup P_3$};
\end{scope}
\end{tikzpicture}
\caption{The three cycles obtained by choosing two branches of a theta. In each panel, the solid colored branches form the selected cycle and the omitted branch is shown in gray.}
\label{fig:theta-choices}
\end{figure}

\begin{samepage}
\begin{lemma}\label{lem:theta}
Let $\mathcal S$ be the family of cycles indexed by a theta support,
and write $s=|\mathcal S|$. When branch $j\in\{1,2,3\}$ is omitted, let
$\mathcal S_j$ be the family just described, and put
\[
 \ell_j=M(\mathcal S)-M(\mathcal S_j),
 \qquad d_j=s-|\mathcal S_j|.
\]
Then
\begin{equation}\label{eq:thetasum}
 \sum_{j=1}^3\ell_j\le M(\mathcal S),
 \qquad \sum_{j=1}^3d_j=2s-1\ge\frac32s.
\end{equation}
For each choice $j$, the decrease $d_j$ is positive, the selected
family is disjoint from every
original cycle outside the support, and the minimum order of its cycles
is at least the minimum order in $\mathcal S$.
\end{lemma}
\end{samepage}

\begin{proof}
On either branch cycle, consider, for each pair of ports, the arc between them that contains the third port. The longer arc used by the new cycle is at least as long as this arc. These three comparison arcs contain every edge twice, so the branch
cycle contributes at least twice its order to the sum of the three
output lengths. An internal cycle makes the same total
contribution or more, since it is retained whole once and contributes
a longer arc in each of the other two choices. The transition paths
only increase the output lengths, so
$\sum_jM(\mathcal S_j)\ge2M(\mathcal S)$, proving the first inequality.

\bigskip

Let $b_j$ be the number of internal cycles on branch $j$.
Then $\sum_jb_j=s-2$ and $|\mathcal S_j|=1+b_j$. Hence
$\sum_jd_j=3s-3-(s-2)=2s-1$. Since $b_j\le s-2$,
each $d_j\ge1$ and $2s-1\ge3s/2$ because $s\ge2$.

The disjointness assertions follow from the definition of a realization, while
the minimum order is preserved because each new cycle contains a longer arc
from each of two branch cycles. When all input cycles have order at least
$L$, these arcs together have at least $L$ edges, and every retained
original cycle already has order at least $L$.
\end{proof}

\subsection{Two inequalities}
The theta and smoothing estimates will be combined using the constants
\begin{equation}\label{eq:rhozet}
 \rho=\frac{86}{95},\qquad \zeta=\frac{99}{2500},
 \qquad \sigma(a)=\frac{1-5a}{4}.
\end{equation}

\begin{samepage}
\begin{lemma}\label{lem:numerical}
The constants in \eqref{eq:rhozet} satisfy
\begin{equation}\label{eq:thetaparameter}
 \zeta^{\rho-1}<\frac{3\rho}{2}.
\end{equation}
For every $0\le a\le\zeta$ and $\sigma(a)\le s\le1-a$,
\begin{equation}\label{eq:smoothparameter}
 \left(1-a-\frac s2\right)^\rho
 +\frac23a^\rho+\frac13(a+s)^\rho\le1.
\end{equation}
\end{lemma}
\end{samepage}

\begin{proof}
The first inequality follows from the integer comparison
\begin{equation}\label{eq:exacttheta}
 2500^9\,95^{95}<99^9\,129^{95}.
\end{equation}
To prove the second, define
\[
 f_a(s)=1-\left(1-a-\frac s2\right)^\rho
             -\frac23a^\rho-\frac13(a+s)^\rho.
\]
Put $B=1-a-s/2$ and $T=a+s$. On the stated domain,
$T\ge(1-a)/4>6/25$ and $0<B\le1$. Since
\[
 \frac BT<\frac{25}{6}<\left(\frac32\right)^{95/9},
\]
we have
\[
 \frac1\rho\frac{\partial f_a}{\partial s}
 =\frac12B^{-9/95}-\frac13T^{-9/95}>0.
\]
It is therefore enough to set $s=\sigma(a)$.

Let $g(a)=f_a(\sigma(a))$. At this value of $s$,
\[
 B(a)=\frac{7-3a}{8},\qquad T(a)=\frac{1-a}{4}.
\]
For $0<a\le\zeta$, differentiation gives
\begin{equation}\label{eq:gderivative}
 \frac{g'(a)}\rho
 =\frac38B(a)^{-9/95}+\frac1{12}T(a)^{-9/95}
                  -\frac23a^{-9/95}<0.
\end{equation}
Indeed, $B(a),T(a)>6/25$ and $9/95<1/10$, so both of their negative
powers are less than $(25/6)^{1/10}<6/5$. Also $a^{-9/95}\ge1$.
The right side of \eqref{eq:gderivative} is therefore less than
$11/20-2/3<0$. The inequality $(25/6)^{1/10}<6/5$ follows by raising
to the tenth power.

The monotonicity of $g$ reduces the remaining estimate to
$g(\zeta)>0$, for which we use the values
\[
 \sigma(\zeta)=\frac{401}{2000},\qquad
 B(\zeta)=\frac{17203}{20000},\qquad
 T(\zeta)=\frac{2401}{10000}.
\]
The following rational upper bounds suffice:
\begin{equation}\label{eq:exactbounds}
 \begin{aligned}
 \left(\frac{17203}{20000}\right)^{86/95}
     &<\frac{872515}{10^6},\\
 \left(\frac{99}{2500}\right)^{86/95}
     &<\frac{53771}{10^6},\\
 \left(\frac{2401}{10000}\right)^{86/95}
     &<\frac{274848}{10^6}.
 \end{aligned}
\end{equation}
Each bound follows by raising both sides to the ninety-fifth power and clearing denominators. Taking their weighted sum gives
\[
 B(\zeta)^\rho+\frac23\zeta^\rho+\frac13T(\zeta)^\rho
 <\frac{599987}{600000}<1.
\]
Since $g$ is decreasing, $g(a)\ge g(\zeta)>0$ throughout
$0<a\le\zeta$, and continuity extends this conclusion to $a=0$.
\end{proof}

\subsection{One replacement round}

\begin{samepage}
\begin{proposition}\label{prop:round}
Let $t>14$, and let $G$ be a $t$-tough graph containing a family
$\CC$ of $q\ge2$ vertex-disjoint cycles, each of order at least $B_t$.
Then $G$ contains a nonempty family $\CC'$ of fewer than $q$
vertex-disjoint cycles satisfying
\begin{equation}\label{eq:roundconclusion}
 \frac{M(\CC')}{M(\CC)}\ge
 \left(\frac{|\CC'|}{q}\right)^\rho.
\end{equation}
Every cycle in $\CC'$ has order at least the minimum order of a cycle
in $\CC$.
\end{proposition}
\end{samepage}

\begin{proof}
Use \cref{lem:ports} with $b=7$ to obtain a transition packing whose auxiliary multigraph $Q$ has minimum degree at least seven. Choose an inclusion-maximal family
of vertex-disjoint thetas in $Q$, and let $X$ be the union of their
supports. Put
\[
 q=|\CC|,\qquad M=M(\CC),\qquad
 a=\frac{|X|}{q},\qquad
 y=\frac{1}{M}\sum_{v\in X}|C_v|.
\]
We have $a>0$, since otherwise $Q$ would be a cactus and
\cref{lem:cactusbound} would contradict its minimum degree.

\bigskip

First suppose $a\ge\zeta$. On each theta support, choose an omitted
branch independently and uniformly, and perform the replacement in
\cref{lem:theta}. Retain every original cycle outside $X$, and denote the resulting family by $\CC'$.
Put $\ell=M-M(\CC')$ and $d=q-|\CC'|$, and let $\mathbb E$ denote expectation over the choices of omitted branches. By \eqref{eq:thetasum},
\[
 \mathbb E\ell\le\frac{yM}{3},\qquad
 \mathbb E d\ge\frac{aq}{2}.
\]
If $y\le3\rho a/2$, then
$\mathbb E(\ell-\rho Md/q)\le0$. Some choice therefore satisfies
$\ell\le\rho Md/q$. For this choice, concavity gives
\[
 \frac{M(\CC')}{M}\ge1-\rho\frac dq
 \ge\left(1-\frac dq\right)^\rho
 =\left(\frac{|\CC'|}{q}\right)^\rho.
\]
Every theta replacement decreases the cycle count, and all selected
families are disjoint because their supports and prescribed transitions are disjoint. If instead $y>3\rho a/2$,
then \eqref{eq:thetaparameter} and the monotonicity of $a^{\rho-1}$ give
$y>a^\rho$. Retaining only the original cycles indexed by $X$,
this is a proper subfamily: when $a=1$, we have $y=1<3\rho/2$, so $a\neq 1$.

Now suppose $0<a<\zeta$. Maximality of the theta family implies that
$Q-X$ contains no theta. Apply \cref{lem:smoothing}, and let
$Y=\bigcup_j S_j$ be the set of indices used by the cycles $K_j$ supplied by the lemma.
Write
\[
 s=\frac{|Y|}{q},\qquad
 u=\frac{1}{M}\sum_{v\in X}|C_v|,\qquad
 v=\frac{1}{M}\sum_{w\in Y}|C_w|.
\]
By \eqref{eq:smoothingcount}, $s>\sigma(a)>0$. The sets $X$ and $Y$
are disjoint, so $a+s\le1$.

\bigskip

If $u\ge a^\rho$, retain only the original cycles indexed by $X$.
If $a+s<1$ and $u+v\ge(a+s)^\rho$, retain only those indexed by
$X\cup Y$. Each choice is a nonempty proper subfamily satisfying
\eqref{eq:roundconclusion}. We may therefore assume
\begin{equation}\label{eq:pooledmass}
 u<a^\rho,\qquad u+v\le(a+s)^\rho.
\end{equation}
The second inequality is automatic when $a+s=1$, in that case we do
not select the whole original family.

\bigskip

Take all the cycles $K_j$ supplied by \cref{lem:smoothing},
discard the original cycles indexed by $X$, and retain every original
cycle outside $X\cup Y$. The selected cycles are disjoint by the last assertion of that lemma. Each new cycle accounts for
at least two members of $Y$, and the new cycles retain at least two
thirds of the total length on $Y$. Thus
\begin{align*}
 \frac{|\CC'|}{q}&\le1-a-\frac s2,\\
 \frac{M(\CC')}{M}&\ge1-u-\frac v3
 =1-\frac23u-\frac13(u+v)\\
 &>1-\frac23a^\rho-\frac13(a+s)^\rho
 \ge\left(1-a-\frac s2\right)^\rho.
\end{align*}
The final step uses \cref{lem:numerical}, giving the length estimate
in \eqref{eq:roundconclusion}. Since $s>0$, at least one cycle $K_j$ is present, so the selected family is nonempty and the bound on
its count shows that it has strictly fewer than $q$ cycles.

\bigskip

The minimum cycle order is preserved in the theta case by
\cref{lem:theta} and in the smoothing case because each cycle $K_j$ retains at least two thirds of the sum of at least two input
orders, giving at least four thirds of the input cycles minimum. Together
with the retained original cycles, the selected cycles therefore satisfy
the required minimum order.
\end{proof}

\section{Iteration and the circumference bound}\label{sec:iteration}

\begin{samepage}
\begin{proposition}\label{prop:iterate}
Let $t>14$, and let $G$ be a $t$-tough graph containing a nonempty
family $\CC_0$ of $q_0$ vertex-disjoint cycles, each of order at least
$B_t$. If $M_0=M(\CC_0)$, then the circumference $F$ of $G$ satisfies
\begin{equation}\label{eq:telescope}
 F\ge M_0q_0^{-\rho}.
\end{equation}
\end{proposition}
\end{samepage}

\begin{proof}
When $q_0=1$, the family already consists of a single cycle of order
$M_0\le F$. For $q_0>1$, apply \cref{prop:round} repeatedly, writing
$q_i=|\CC_i|$ and
$M_i=M(\CC_i)$. The minimum cycle order stays at least $B_t$, so the
proposition remains applicable. Every round strictly decreases the
positive integer $q_i$, and the process ends with $q_N=1$. Multiplying
\eqref{eq:roundconclusion} telescopes and gives
\[
 M_N\ge M_0\prod_{i=0}^{N-1}
 \left(\frac{q_{i+1}}{q_i}\right)^\rho
 =M_0q_0^{-\rho}.
\]
The final family contains one cycle, so $M_N\le F$.
\end{proof}

\begin{proof}[Proof of \Cref{thm:main}]
\label{proof:mainproof}
Fix $t>14$ and use $r_t,A_t$ from \eqref{eq:rt}. Define
\begin{equation}\label{eq:delta}
 \delta_t=\frac{1-\rho}{1+(r_t+1)(1-\rho)}
 =\frac{9}{95+9(r_t+1)}.
\end{equation}
Let $G$ have order $n\ge3$ and circumference $F$.
If \eqref{eq:initiallarge} holds, then
\[
 F>a_t n^{1/(r_t+1)},\qquad
 a_t=(2A_t)^{-1/(r_t+1)}.
\]
Since $\delta_t<1/(r_t+1)$ and $n\ge1$, this implies
$F\ge a_tn^{\delta_t}$.

Otherwise \cref{prop:initial} gives a family $\CC_0$ with
\[
 q_0\ge\frac{n}{2A_tF^{r_t+1}},
 \qquad M_0\ge L_{r_t}q_0.
\]
By \cref{prop:iterate},
\[
 F\ge L_{r_t}q_0^{1-\rho}
 \ge L_{r_t}\left(\frac{n}{2A_tF^{r_t+1}}\right)^{1-\rho}.
\]
Rearranging yields
\[
 F\ge b_t n^{\delta_t},\qquad
 b_t=L_{r_t}^{1/[1+(r_t+1)(1-\rho)]}(2A_t)^{-\delta_t}.
\]
For $t=15$, \eqref{eq:15constants} gives $r_{15}=8$, so
$\delta_{15}=9/(95+81)=9/176$, yielding the bound in
Theorem \ref{thm:main}.
\end{proof}

\begin{samepage}
\begin{corollary}\label{cor:larget}
For every fixed $t\ge266/5$, there is a constant $c_t>0$ such that
every $t$-tough graph of order $n\ge3$ has circumference at least
$c_t n^{9/131}$.
\end{corollary}
\end{samepage}

\begin{proof}
The inequality $t\ge266/5$ is equivalent to $14t/(t-14)\le19$.
Since $B_t>14>L_2$ and $L_3=19$, we have $r_t=3$, and substitution
in \eqref{eq:delta} gives the exponent $9/131$.
\end{proof}

\section*{Acknowledgements}\label{sec:acknowledgements}
This work was carried out during the first authors two-month stay at the University of California, San Diego.
\begin{comment}
The author is deeply grateful to Jacques Verstraete for very many helpful and fruitful discussions, without which this work would not have been possible.
\end{comment}


\begin{thebibliography}{99}

\bibitem{ToughnessSurvey}
D. Bauer, H.~J. Broersma, and E. Schmeichel,
\emph{Toughness in graphs---A survey},
Graphs and Combinatorics \textbf{22} (2006), no.~1, 1--35.
\href{https://doi.org/10.1007/s00373-006-0649-0}{doi:10.1007/s00373-006-0649-0}.

\bibitem{BBV}
D. Bauer, H.~J. Broersma, and H.~J. Veldman,
\emph{Not every $2$-tough graph is Hamiltonian},
Discrete Applied Mathematics \textbf{99} (2000), nos.~1--3, 317--321.
\href{https://doi.org/10.1016/S0166-218X(99)00141-9}{doi:10.1016/S0166-218X(99)00141-9}.

\bibitem{BohmePlanar}
T. B\"ohme, H.~J. Broersma, and H.~J. Veldman,
\href{https://ris.utwente.nl/ws/files/6852394/Bohme96toughness.pdf}{\emph{Toughness and longest cycles in $2$-connected planar graphs}},
Journal of Graph Theory \textbf{23} (1996), no.~3, 257--263.


\bibitem{BHVJ}
H.~J. Broersma, J. van den Heuvel, H.~A. Jung, and H.~J. Veldman,
\emph{Long paths and cycles in tough graphs},
Graphs and Combinatorics \textbf{9} (1993), 3--17.
\href{https://doi.org/10.1007/BF01195323}{doi:10.1007/BF01195323}.

\bibitem{ChenChordal}
G. Chen, M.~S. Jacobson, A.~E. K\'ezdy, and J. Lehel,
\href{https://onlinelibrary.wiley.com/doi/10.1002/(SICI)1097-0037(199801)31:1<29::AID-NET4>3.0.CO;2-M}{\emph{Tough enough chordal graphs are Hamiltonian}},
Networks \textbf{31} (1998), no.~1, 29--38.

\bibitem{CY} G. Chen, X. Yu, \emph{Long Cycles in 3-Connected Graphs},
Journal of Combinatorial Theory, Series B,
Volume 86, Issue 1, 2002, Pages 80--99,



\bibitem{Chvatal}
V. Chv\'atal,
\emph{Tough graphs and hamiltonian circuits},
Discrete Mathematics \textbf{5} (1973), no.~3, 215--228.
\href{https://doi.org/10.1016/0012-365X(73)90138-6}{doi:10.1016/0012-365X(73)90138-6}.

\bibitem{EJKS}
H. Enomoto, B. Jackson, P. Katerinis, and A. Saito,
\emph{Toughness and the existence of $k$-factors},
Journal of Graph Theory \textbf{9} (1985), no.~1, 87--95.
\href{https://doi.org/10.1002/jgt.3190090106}{doi:10.1002/jgt.3190090106}.

\bibitem{GLSTY} C. Groenland, S. Longbrake, R. Steiner, J. Turcotte and L. Yepremyan, {\emph Longest cycles in vertex-transitive
and highly connected graphs}, Bull. Lond. Math. Soc. 57 (2025), no. 10, 2975--2990.
\href{https://doi.org/10.1112/blms.70134}{doi:10.1112/blms.70134}.


\bibitem{Hasanvand}
M. Hasanvand,
\emph{Spanning trees and spanning closed walks with small degrees},
Discrete Mathematics \textbf{345} (2022), no.~10, article~112998.
\href{https://doi.org/10.1016/j.disc.2022.112998}{doi:10.1016/j.disc.2022.112998}.

\bibitem{JacksonWormald}
B. Jackson and N.~C. Wormald,
\emph{$k$-walks of graphs},
Australasian Journal of Combinatorics \textbf{2} (1990), 135--146.

\bibitem{JW} B. Jackson and N. C. Wormald, \emph{Longest cycles in 3‑connected graphs of bounded maximum degree}, In R. Rees (Ed.), Graphs, Matrices and Designs (1992) (Vol. 139, pp. 237--254).

\bibitem{JungWittmann}
H.~A. Jung and P. Wittmann,
\href{https://ris.utwente.nl/ws/files/283018877/Jung1999longest.pdf}{\emph{Longest cycles in tough graphs}},
Journal of Graph Theory \textbf{31} (1999), no.~2, 107--127.


\bibitem{KabelaLong}
A. Kabela,
\emph{Long paths and toughness of $k$-trees and chordal planar graphs},
Discrete Mathematics \textbf{342} (2019), no.~1, 55--63.
\href{https://doi.org/10.1016/j.disc.2018.08.017}{doi:10.1016/j.disc.2018.08.017}.

\bibitem{KabelaKaiser}
A. Kabela and T. Kaiser,
\emph{$10$-tough chordal graphs are Hamiltonian},
Journal of Combinatorial Theory, Series B \textbf{122} (2017), 417--427.
\href{https://doi.org/10.1016/j.jctb.2016.07.002}{doi:10.1016/j.jctb.2016.07.002}.

\bibitem{LesniakSchmeichel}
L. Lesniak and E. Schmeichel,
\emph{The circumference of $2$-tough graphs},
Discrete Applied Mathematics \textbf{389} (2026), 280--284.
\href{https://doi.org/10.1016/j.dam.2026.03.052}{doi:10.1016/j.dam.2026.03.052}.

\bibitem{LiuYuZhang}
Q. Liu, X. Yu, and Z. Zhang,
\emph{Circumference of $3$-connected cubic graphs},
Journal of Combinatorial Theory, Series B \textbf{128} (2018), 134--159.
\href{https://doi.org/10.1016/j.jctb.2017.08.008}{doi:10.1016/j.jctb.2017.08.008}.


\bibitem{MaZhao} J. Ma, Z. Zhao, \emph{Linear circumference in vertex-transitive graphs}, 
\href{https://arxiv.org/abs/2610.02053}. 

\bibitem{Mader}
W. Mader,
\emph{\"Uber die Maximalzahl kreuzungsfreier $H$-Wege},
Archiv der Mathematik \textbf{31} (1978), 387--402.

\bibitem{Menger}
K. Menger,
\emph{Zur allgemeinen Kurventheorie},
Fundamenta Mathematicae \textbf{10} (1927), 96--115.
\href{https://doi.org/10.4064/fm-10-1-96-115}{doi:10.4064/fm-10-1-96-115}.

\bibitem{Pap}
G. Pap,
\href{https://egres.elte.hu/tr/egres-06-17.pdf}{\emph{Mader matroids are gammoids}},
Egerv\'ary Research Group, Technical Report TR-2006-17, 2006.


\bibitem{Rado}
R. Rado,
\emph{A theorem on independence relations},
The Quarterly Journal of Mathematics, Oxford Series \textbf{13} (1942), 83--89.
\href{https://doi.org/10.1093/qmath/os-13.1.83}{doi:10.1093/qmath/os-13.1.83}.

\bibitem{SV} L. Segal, J. Verstrae\"{e}te, \emph{Logarithmic circumference in tough graphs}, Preprint (2026).


\bibitem{SchrijverPaths}
A. Schrijver,
\href{https://homepages.cwi.nl/~lex/files/mader5.pdf}{\emph{A short proof of Mader's $\mathcal S$-paths theorem}},
Journal of Combinatorial Theory, Series B \textbf{82} (2001), no.~2, 319--321.


\bibitem{SchrijverMader}
A. Schrijver,
\href{https://homepages.cwi.nl/~lex/files/maderq.pdf}{\emph{Is each Mader matroid a gammoid?}},
unpublished note.

\bibitem{SchrijverMatroid}
A. Schrijver,
\emph{Combinatorial Optimization: Polyhedra and Efficiency},
Algorithms and Combinatorics, vol.~24, Springer, Berlin, 2003.

\bibitem{WangFractional}
Z. Wang,
\emph{Toughness Bounds for Fractional Hamiltonicity and Resistance Positivity},
preprint, \href{https://arxiv.org/abs/2609.03412}{arXiv:2609.03412} (2026).

\bibitem{Win}
S. Win,
\emph{On a connection between the existence of $k$-trees and the toughness of a graph},
Graphs and Combinatorics \textbf{5} (1989), 201--205.
\href{https://doi.org/10.1007/BF01788671}{doi:10.1007/BF01788671}.

\end{thebibliography}
\end{document}